\documentclass{article}

\usepackage{mathrsfs}
\usepackage{amsmath}
\usepackage{amssymb}
\usepackage{amsthm}
\usepackage{graphicx}
\usepackage{enumerate}
\usepackage{tikz-cd}
\usepackage{fancyhdr}
\usepackage{pictex}
\usepackage{multirow}
\usepackage{mathtools}
\usepackage{hyperref}

\newtheorem{theorem}{Theorem}[section]
\newtheorem{rev}[theorem]{Revision}

\newtheorem{cor}[theorem]{Corollary}

\newtheorem{definition}[theorem]{Definition}

\newtheorem{lemma}[theorem]{Lemma}

\newtheorem*{th*}{Theorem}
\theoremstyle{remark}
\newtheorem{remark}[theorem]{$\textbf{Remark}$}
\newtheorem{remarks}[theorem]{$\textbf{Remarks}$}
\newtheorem{fact}[theorem]{$\textbf{Fact}$}
\newtheorem{assumption}[theorem]{$\textbf{Assumption}$}

\newcommand{\ag}{\`}
\newcommand{\noi}{\noindent}

\newcommand{\Q}{\mathbb{Q}}

\newcommand{\C}{\mathbb{C}}
\newcommand{\Z}{\mathbb{Z}}
\newcommand{\N}{\mathbb{N}}

\newcommand{\PP}{\mathbb{P}^1}
\newcommand{\QQ}{\mathbb{P}}
\newcommand{\OO}{\mathcal{O}}

\title{Dynamical systems with a parallel tensor}
\author{Jacopo Garofali}
\date{June 2018}

\begin{document}
\maketitle
\begin{abstract}
  We classify those rational maps $f:\PP \to \PP$ for which there exists a contravariant tensor $q$ which is parallel, i.e. such that $f^*q\,//\,q$, by proving that such maps preserve a parabolic orbifold.
\end{abstract}

\section{Introduction}
A holomorphic dynamical system on the Riemann sphere $\PP$
is the data of a rational map $f:\PP\to \PP$.
From the viewpoint of Dynamics, the principal object of interest
is the study of the space of orbits $\PP/f$ under the equivalence
relation generated by \textit{f}, namely $z \sim w$ if and only if
there exist nonnegative integers $n$ and $m$ such that $f^n(z)=f^m(w)$. \\
As one can imagine, this problem is not easily solved since the
quotient space $\PP/f$ does not exist in general,
at least not in a classic sense.
The fact that the dynamical system is not given by the action of a group,
unless the degree of the map \textit{f} is 1,
is one of the main obstacles in the matter.\\
However we can always consider the
simplicial object, $\QQ^{\bullet}_f$, associated to the dynamical
system which in degrees $0$ and $1$ is given by,
$$
\PP \leftleftarrows \coprod_{n\in\N}
\Gamma_{f^n}:=\coprod_{n\in\N}\{(x,f^n(x))\,:\,x\in \PP\}
$$
and we can identify sheaves on $\PP/f$ with
a simplicial sheaf on $\QQ^{\bullet}_f$ in the
sense of \cite[5.I.6]{MR0498552}. Consequently to give
a sheaf on $\QQ^{\bullet}_f$ is equivalent to giving
a sheaf $\mathcal{F}$ on $\PP$ together with a
map of sheaves,
$$
A_f: f^*\mathcal{F}\rightarrow\mathcal{F}
$$
One question that arises naturally in the study of dynamical
systems is whether we are able to define ``objects'' that are
invariant for the dynamics (i.e. global sections of a simplicial sheaf)
and if possible, understand their nature. \\
Specifically given a simplicial sheaf $\mathcal{F}$
as above on $\QQ^{\bullet}_f$, a global section is,
by definition, an element $q \in H^0(\PP,\mathcal{F})$ which is invariant for
the action of \textit{f}, i.e.
$$A_f(f^*q)=q,$$
and we write $q \in H^0(\PP/f,\mathcal{F})$
as a shorthand for  $H^0(\QQ^{\bullet}_f,\mathcal{F})$.
\\ \\

\noi The main purpose of our work is to investigate the existence
of $k$-th differentials on $\PP$ which, after twisting by a locally
constant simplicial line bundle, are invariant for the dynamical
system generated by \textit{f},
and,  to classify them.\\
We have been introduced to this problem while we were studying the work
of Adam L. Epstein \cite{1999math......2158E}.
In his extension of \textit{Infinitesimal Thurston rigidity}
he shows that a rational map $f:\PP \to \PP$ of degree $d>1$,
for which there exists a meromorphic quadratic differential
$q$ with $f^*q=d\,q$ is a ``Latt\ag es map'' \cite{2004math......2147M}.
This led us to ask for which maps
$f:\PP \to \PP$ does there exist a non-zero meromorphic global section $q$ of $\Omega_{\PP}^{\otimes k}$ and a constant $\lambda\in \C^*$ such that:
\begin{equation}\label{assumption}
f^*q=\lambda q
\end{equation}
wherein we employ the standard convention, \cite{MR0498552}, of identifying
$A_f$ of a differential with its image.\\ \\

To interpret \eqref{assumption} in the simplicial language of
\cite{MR0498552}, observe that we have a simplicial local
system $L_{\lambda}$ given by way of the action on the trivial sheaf,
$$
A_f(f^* 1) := \lambda
$$
and we define $\Omega_{\PP}^{\otimes k}(\lambda):=
\Omega_{\PP}^{\otimes k}\otimes_{\OO}L_{\lambda^{-1}}$,
so that a mermorphic differential satisfies \eqref{assumption}
if and only if,
$$q \in H^0(\PP/f,\Omega_{\PP}^{\otimes k}(\lambda)).$$
In any case, whether in the simplicial language or
the more elementary \eqref{assumption} our main theorem is:

\begin{th*}
All the rational maps $f:\PP \to \PP$ of degree $d>1$ for which
there exists a nonzero holomorphic section
$q\in H^0(\PP/f,\Omega_{\PP}^{\otimes k}(\lambda))$ are
(modulo at worst an element of $PGL_2(\C)$ of order 2 or 3)
equivalent to the action of an endomorphism of elliptic curves,
and thus the action of \textit{f} comes from the action
of a group of automorphisms of $\C$.\\
They are all listed in \hyperref[table]{ Table \ref*{table} }.
\end{th*}

\noi \textbf{Acknowledgements}
We owe profound thanks to Michael McQuillan for his unflagging patience in guiding us throughout this work.

\section{A simple case}
The space $Rat(d)$ of all rational maps $f:\PP \to \PP$
with $deg(f)=d$, is never a group unless $d=1$, i.e. the group of automorphisms of the
complex projective line $PGL_2(\C)$.\\
The subgroup generated by $f \in PGL_2(\C)$ is clearly isomorphic
to $\Z$, and it acts on $\PP$ through
\begin{center}
\begin{tikzcd}[row sep=0.6pc, column sep=2.6pc]
\Z \times \PP \arrow[r] & \PP \\
 m \times z \arrow[r, mapsto] & f^m(z)
\end{tikzcd}
\end{center}
Recall the Jordan decomposition of $2\times2$ matrices, to wit:
\begin{rev}\label{theorem0}
Any $f \in PGL_2(\C)$ is conjugated by an element of
$ PGL_2(\C)$ to one of the following:
\begin{enumerate}[1)]
  \item $f(z)=z+\beta,\; \beta \in \mathbb{G}_a$ \;
  (if and only if \textit{f} has only one fixed point);
  \item $f(z)=\alpha z, \; \alpha \in \mathbb{G}_m$ \;
  (if and only if \textit{f} has two distinct fixed points).
\end{enumerate}
\end{rev}

\noi We want to characterize all meromorphic global
sections $q$ of $\Omega_{\PP}^{\otimes k}$ satisfying $f^*q=\lambda q$.
Since \textit{f} is an automorphism, for any $x \in \PP$ we have
$ord_x(q)=ord_{f(x)}(q)$, hence for any $k \in \Z$ the sets
$S_k=\{x\in \PP: ord_x(q)=k\}$ are completely invariant for the dynamics,
i.e. $f^{-1}(S_k)=S_k$.\\
From \ref{theorem0} we deduce easily that in Case 1)
the only finite set which is completely invariant for \textit{f}
is the fixed point $\infty$. We conclude that $\infty$ is the unique
pole of $q$, hence $q(z)=const \cdot dz^k$.
In Case 2) $q$ may have both poles and zeroes. If they are contained
in $\{0,\infty\}$ then $\displaystyle q(z)= const \cdot z^a \left( \frac{dz}{z}\right)^k$,
and the coefficient $a$ is determined by $ord_0(q)$.
Suppose that for some $k\in \Z$ we have
\begin{equation}\label{invariantset}
  S_k \neq \emptyset, \text{ and }  S_k \nsubseteq \{0,\infty\}
\end{equation}
Since for any $x \in S_k$ we have $\alpha^n x=x$ for some $n>1$,
there exists some minimal $n$ such that $f^n=id$.
It follows that $\alpha$ is a primitive $n$-th root of unity,
and as $f^*q=\lambda q$ we see easily that $\lambda=\alpha^j$ for some $j<n$.\\
Note that the action of $f: \mathbb{G}_m \to \mathbb{G}_m $ is a free action,
so the quotient map $p:  \mathbb{G}_m \to \mathbb{G}_m/f$ is canonically a $\mu_n$-torsor.\\
Define
\begin{equation}\label{explicit}
  Q(z)=z^j\left(\frac {dz}{z}\right)^k.
\end{equation}
and note that $Q \in H^0( \mathbb{G}_m, \Omega^{\otimes k})$.\\
Moreover $f^*Q=\lambda Q$, \textit{i.e.} $Q \in H^0(\mathbb{G}_m/f,\Omega^{\otimes k}\otimes L_{\lambda}^{-1})$ where $L_{\lambda}$ denotes the sheaf on $\mathbb{G}_m/f$ given by the action
on the trivial sheaf $A_f(f^*1)=\lambda$.\\
Thus ``multiplication by $Q$'' yields an $f$-invariant
isomorphism of sheaves on $\mathbb{G}_m$
$$
\OO \mathop{\longrightarrow}\limits^{\sim} \Omega^{\otimes k}\otimes L_{\lambda}^{-1}
$$

We deduce the following.
\begin{fact}
  Let $f$ satisfy condition \eqref{invariantset},
  then a meromorphic global section
  $q\in H^0(\mathbb{G}_m/f,\Omega^{\otimes k}\otimes L_{\lambda}^{-1}) $,
  i.e. a meromorphic $k$-th differential $q$ with $f^*q=\lambda q$,
  is necessarily of the form $q(z)=g(z^n)Q(z)$, where $g$ is a meromorphic function and $Q$ is given by \eqref{explicit}.
\end{fact}

\section{Dynamical systems on the Riemann sphere with a parallel tensor}
In the holomorphic category, a non-unit endomorphism of $\PP$
is a rational map $f: \PP \to \PP$ of degree $d>1$.
We denote by $\Omega_{\PP}$ the sheaf of holomorphic
differential forms on $\PP$, given by the canonical
action $dz \to f'(z)\,dz$, and by $\Omega_{\PP}^{\otimes k}$
its $k$-th tensor power. \\
Let us suppose from now on that \textit{f} verifies the following Assumption:
\begin{assumption}\label{assum}
 There exists $k\in \N^*$ and a global meromorphic section of
 $\Omega_{\PP}^{\otimes k}$, which we will denote by \textit{q},
 such that $f^*q=\lambda q$, for some $\lambda \in \C^*$ .
\end{assumption}
\noi Let `\textit{z}' be a local coordinate around a point
$x \in \PP$, we can write \textit{q} in the form $q=q(z)\,dz^k$,
where $q(z)$ is a meromorphic function of \textit{z}.
For any $y \in f^{-1}(x)$, let `\textit{s}' be the local
coordinate around \textit{y} such that the map \textit{f} in this coordinates takes the form $s \mapsto s^n$, where $n:=deg_y(f)$.
Within this notation we have $f^*q=q(s^n)\,(ns^{n-1}ds)^k$ and
it follows easily that
\begin{equation}\label{eq:1}
  ord_y(f^*q)=deg_y(f)(ord_x(q) +k) -k
\end{equation}.\\
Now \textit{\textbf{Assumption}} \ref{assum} clearly implies that $ord_y(f^*q)=ord_y(q)$, so we obtain
\begin{equation}\label{eq:2}
  \forall x \in \PP,  \qquad  ord_x(q)=deg_x(f)(ord_{f(x)}(q) +k) -k
\end{equation}

\subsection{Considerations on zeroes and poles}
In this section we are going to show that \eqref{eq:2} constrains the number of zeroes and poles of \textit{q}.
\begin{lemma}\label{lemma1}
  Let  $f: \PP \to \PP$ be a rational map of degree $d>1$ which satisfies \textbf{Assumption} \ref{assum}.\\ Then \textit{q} has no zeroes, that is the set $Z:=\{ x \in \PP : ord_x(q) >0 \}$ is the empty set.
\end{lemma}
\begin{proof}
Let us define the divisor $\mathcal{Z}:= \sum\limits_{x \in Z} ord_x(q)\,x$ on $\PP$,
supported on the zeroes of \textit{q}.\\
Note that the statement of our Lemma is equivalent to
\begin{equation}\label{eq:lemma1}
  deg(\mathcal{Z}) =0
\end{equation}
We claim that \textit{Z} is a backward invariant set for the dynamics, i.e. $f^{-1}(Z) \subset Z$.
Given a zero $x \in Z$ of \textit{q} with $ord_x(q)=e>0$, we see from \eqref{eq:2} that for any
$y \in f^{-1}(x)$, setting $n:=deg_y(f)$, we have $ord_y(q)=ne + k(n-1)>0$, i.e. $y \in Z$.\\
Observe now that, from the definition of $f^*\mathcal{Z}$, we have
\begin{equation}\label{eq:lemma1bis}
  ord_y(f^*q)=ord_y(f^*\mathcal{Z}) + k(n-1) \geq ord_y(f^*\mathcal{Z})
\end{equation}
Thus summing \eqref{eq:lemma1bis} over all $y \in f^{-1}(Z)$, we obtain
\begin{equation}\label{eq:lemma1tris}
  \sum\limits_{y \in f^{-1}(Z)} ord_y(f^*q) \geq  \sum\limits_{y \in f^{-1}(Z)} ord_y(f^*\mathcal{Z})=deg(f^*\mathcal{Z})
\end{equation}
The left hand side of \eqref{eq:lemma1tris} is obviously less than
or equal to $$\sum\limits_{y \in Z} ord_y(f^*q)= \sum\limits_{y \in Z}
ord_y(q)=deg(\mathcal{Z})$$ since we are summing nonnegative numbers over a smaller set.\\
Finally we obtain $deg(\mathcal{Z}) \geq deg(f^*\mathcal{Z})=d\cdot deg(\mathcal{Z})$,
which implies \eqref{eq:lemma1} as we assumed $d>1$.
\end{proof}

\begin{lemma}\label{lemma2}
 Let  $f: \PP \to \PP$ be a rational map of degree $d>1$, which satisfies \textbf{Assumption} \ref{assum}.\\
 If $x \in \PP$ is a pole of \textit{q}, then
 \begin{equation}\label{eq:lemma2}
   -k \leq ord_x(q) \leq - \frac k 2
 \end{equation}
 \end{lemma}
\begin{proof}
We first prove the left inequality which is equivalent to
\begin{equation}\label{eq:lemma2bis}
  \{ x \in \PP: \; ord_x(q)  < -k \}=\emptyset
\end{equation}
In order to show this let us consider the following divisor on $\PP$
supported on $P_k=\{ x \in \PP: \; ord_x(q)  \leq -k \}$
$$
\mathcal{P}_k:= \sum \limits_{x \in P_k } (-ord_x(q) -k)\,x
$$
We claim that $deg(\mathcal{P}_k)=0$ from which
\eqref{eq:lemma2bis} follows immediately.\\
Note that equation \eqref{eq:2} implies that $P_k$
is a backward invariant set.
Moreover we have $-ord_y(q) -k= deg_y(f)(-ord_x(q)-k) =
ord_y(f^* \mathcal{P}_k)$, so that,
 \begin{align*}
    deg(\mathcal{P}_k)= \sum \limits_{x \in P_k} (-ord_x(q) -k)
    \geq & \sum \limits_{y \in f^{-1}(P_k) } (-ord_y(q) -k) = \\
     \sum \limits_{y \in f^{-1}(P_k) } ord_y(f^* \mathcal{P}_k)=
     deg(f^* \mathcal{P}_k)= & d \cdot deg(\mathcal{P}_k)
 \end{align*}
Thus just as $deg(\mathcal{P}_m) \geq 0$ and we assumed $d>1$,
it follows that $deg(\mathcal{P}_k)=0$.\\
Let us prove now the right hand side of \eqref{eq:lemma2}, i.e.
\begin{equation}\label{eq:lemma2tris}
  \{ x \in \PP : -\frac k 2 < ord_x(q) <0 \} = \emptyset
\end{equation}
Let $x \in \PP$ be a pole of \textit{q} of order $m:= -ord_x(q) \neq k$ and let $y \in f^{-1}(x)$.
Note that from \eqref{eq:lemma2bis} we have $m < k$
and also, from \ref{lemma1}, that $ord_y(q) \leq 0$. \\
Consequently, from
\eqref{eq:2} we obtain
\begin{equation}\label{bound}
  deg_y(f) \leq \frac{k}{k-m}
\end{equation}
Now, as $\displaystyle \frac{k}{k-m} <2$, $\forall\; 0<m< \frac k 2$, we deduce
that $deg_y(f)=1$. Thus we see again from \eqref{eq:2} that
the set of poles of \textit{q} of order $m$, with $0<m< \frac k 2$,
is backward invariant. \\
For any such integer $m$ let us consider the divisor on
$\PP$ given by
$$
\mathcal{P}_m:= \sum \limits_{x:\; ord_x(q) = -m} m\,x
$$
It follows from the considerations above that
$f^*\mathcal{P}_m \subset \mathcal{P}_m $,
and therefore \\
$deg(f^*\mathcal{P}_m) \leq deg(\mathcal{P}_m)$.
From $deg(f^*\mathcal{P}_m)= d\, deg(\mathcal{P}_m)$ and $d>1$,
we deduce that $deg(\mathcal{P}_m)=0$, which implies \eqref{eq:lemma2tris}.
\end{proof}

\subsection{Main Lemma: the dynamical system preserves a parabolic orbifold}
In this section we discuss the main consequence
of \textit{\textbf{Assumption}} \ref{assum},
\textit{i.e.} the existence of an orbifold or eventually an orbifold with boundary
that is preserved by \textit{f}.
We refer to \cite{MR1435975} and to \cite{MR1251582}
for a formal definition of an orbifold.
Observe that under our hypothesis, given any ramification point
$x\in Ram_f$, its image $f(x)$ is a pole of \textit{q},
since Lemma \ref{lemma1} and \eqref{eq:2} imply that
$ord_{f(x)}(q)< ord_x(q) \leq 0$, i.e  the order is decreasing.\\
Thus a map \textit{f} satisfying \textit{\textbf{Assumption}} \ref{assum}
must necessarily be a post-critically finite map,
\textit{i.e.} the post-critical set of \textit{f},
$\mathcal{P}_f=\bigcup\limits_{n>0}f^n\left(Ram_f\right)$ is finite.
\begin{definition}\label{def1}
Let $f:\PP \to \PP$ be a rational map of degree $d>1$.
We say that \textit{f} preserves an orbifold,
or also that \textit{f} lifts to a map of
orbifolds, if there exists a function $\nu:\PP \to \N^* \cup \{\infty\}$
satisfying the following conditions:

\begin{equation}\label{eq:3}
\begin{array}{ll}
&1)\nu(x)=1 \mbox{  if } x \notin P_f  \\
&2)\forall x \in  \PP,\, \nu(x) \mbox{ divides } \nu(y)deg_y(f) \mbox{ whenever } y \in f^{-1}(x)
\end{array}
\end{equation}
We shall denote by $\nu_f$ the smallest among all
functions $\nu$ satisfying condition \eqref{eq:3}.
Also, we denote by $\OO=(\PP, \nu_f)$ the orbifold
preserved by \textit{f} and we shall refer to it, for brevity's sake,
through the string $(\nu_f(x)\mathop{)}_{x \in P_f}$.
\end{definition}
By definition if $\nu_f$ takes the value $\infty$
we will say that $\OO$ is an orbifold with boundary, which
we will denote by $(\OO, D)$ \textit{i.e.}
$D \subset \OO$ is the set of singular points of weight $\infty$.
We can associate to $\OO$ its Euler characteristic
$$
\chi(\OO)=2-\sum\limits_{x\in P_f} \left(1-\frac{1}{\nu_f(x)}\right)
$$
which is well defined and extended in the obvious
sense if $\nu_f$ takes value $\infty$.
We shall call an orbifold $\OO$ \textit{hyperbolic}
if $\chi(\OO)<0$, \textit{parabolic} if $\chi(\OO)=0$ and \textit{elliptic} otherwise.
We enunciate here the main result of our work:

\begin{lemma}\label{mainlemma}
Let $f:\PP \to \PP$ be a rational map of degree $d>1$
which satisfies \textbf{Assumption} \ref{assum}.
Then \textit{f} preserves one of the following
parabolic orbifolds (resp. orbifold with boundary):

\begin{enumerate}[(i)]
  \item $(\infty,\infty)$
  \item $(2,2,\infty)$
  \item $(2,2,2,2)$
  \item $(3,3,3)$
  \item $(2,4,4)$
  \item $(2,3,6)$
\end{enumerate}
\end{lemma}

\begin{proof}

It is a fact that, for $q$ a meromorphic global section
of the sheaf $\Omega_{\PP}^{\otimes k}$, we have
\begin{equation}\label{eq:deg}
\sum\limits_{x \in \PP} ord_x(q)=  deg(\Omega_{\PP}^{\otimes k})=-2k
\end{equation}
We shall call $n_i$ the number of poles of \textit{q}
of order $i \geq0$, i.e. $n_i=\#\{x \in \PP : ord_x(q)=-i\}$.
Recall that in Lemma \ref{lemma1} and \ref{lemma2} we proved that
$n_i=0$ $\forall i <\frac{k}{2}$, hence setting $\alpha_i=\frac{i}{k}$,
equation \eqref{eq:deg} becomes:
\begin{equation}\label{eq:poles}
  \sum_{i=i_0}^{k} n_i \alpha_i =2
\end{equation}
where $i_0= \big \lceil \frac{k}{2}\big \rceil$ is
the smallest integer greater than $\frac{k}{2}$.\\
In the first pace we discuss the only two solutions
of \eqref{eq:poles} with $n_k \neq 0$.
One is $n_k=2$, which implies $n_i=0\;\forall i \neq k$,
while the other, which may occur only if $k$ is even, is $n_k=1$ and $n_{i_0}=2$.
We proved in \ref{lemma2} that the set of poles of \textit{q} of order
\textit{k} is a complete invariant set for the dynamics,
hence in the first case if we define
$$
\nu_f(x)=
\begin{cases}
  \infty, & \mbox{if } ord_x(q)=k;\\
  1, & \mbox{otherwise}.
\end{cases}
$$
it is clear that $\nu_f$ satisfies condition \eqref{eq:3},
i.e \textit{f} preserves an orbifold with boundary of type \textit{(i)}.\\
If the second case occurs, let $\infty$ be the pole of \textit{q} of order $k$,
which is a fixed point of \textit{f} and let $P=\{p_1,p_2\}$ be the other poles
of \textit{q}, necessarily of order $\frac{k}{2}$. \\
Recall from \eqref{bound} that given a pole $x$ of $q$ with $ord_x(q)=-n$ and given any $y\in f^{-1}(x)$, we have
\begin{equation}\label{eq:bound}
  deg_y(f)\leq \frac{1}{1- \alpha_n}
\end{equation}
Thus for any $y \in f^{-1}(p_i),\,i=1,2$
we have $deg_y(f) \leq 2$. From \eqref{eq:2}
we see that if $deg_y(f)=1$ then $y\in P$,
and if $deg_y(f)=2$ then $ord_y(q)=0$. We define
\begin{equation}\label{22inf}
\nu_f(x)=
\begin{cases}
  2, & \mbox{if } x \in P; \\
  \infty, & \mbox{if } x=\infty;\\
  1, & \mbox{otherwise}.
\end{cases}
\end{equation}
It follows from the discussion above that $\nu_f$
satisfies condition \eqref{eq:3}, so \textit{f}
preserves an orbifold with boundary of type \textit{(ii)}.\\
Our aim is now to show that any solution of \eqref{eq:poles}
with $n_k=0$ corresponds to one of the orbifold $(iii)-(vi)$ listed above.
Note that we must have $n_i\leq4, \;\forall i$ and we have $n_i=4$
for some \textit{i} if and only if \textit{k} is even and $i=\frac 1 2$.
In this case let $P=\{p_1,p_2,p_3,p_4\}$ be the set of poles of
\textit{q}, each pole having order $\frac k 2$ and let us define
\begin{equation}\label{2222}
  \nu_f(x)=
\begin{cases}
  2, & \mbox{if } x \in P; \\
  1, & \mbox{otherwise}.
\end{cases}
\end{equation}
As we have seen before, for any $y\in f^{-1}(P)$
the only possibilities for $deg_y(f)$ are 1 or 2,
meaning that respectively, $y\in P$ or $ord_y(q)=0$.
Hence $\nu_f$ satisfies condition \eqref{eq:3},
i.e. \textit{f} preserves the orbifold \textit{(iii)}.\\
Suppose now that in equation
\eqref{eq:poles} we have $n_i=3$ for some \textit{i},
then necessarily $\frac 1 2 <\alpha_i \leq \frac{2}{3}$.
Note that the latter inequalities cannot be strict,
since otherwise there would exist an index $j\neq i$ such that $n_j\neq 0$
and consequently $0<n_j \alpha_j=2 - 3\alpha_i < \frac 1 2$ which is impossible.
Therefore $n_i=3$ for some \textit{i} if and only if $\alpha_i=\frac{2}{3}$
(note that it makes sense only if $k\equiv 0\;(mod\, 3)$).
In this case let $P=\{p_1,p_2,p_3\}$ be the set of poles of \textit{q},
each of order $\frac 2 3 k$ and let us define
 \begin{equation}\label{333}
 \nu_f(x)=
\begin{cases}
  3, & \mbox{if } x \in P; \\
  1, & \mbox{otherwise}.
\end{cases}
\end{equation}
It follows from \eqref{eq:bound} that for any $y\in f^{-1}(P)$,
$deg_y(f) \leq 3$, but $deg_y(f)=2$ implies $ord_y(q)=-\frac k 3 $
which is impossible, hence the only possibilities for the local degree
of $f$ at $y$ are 1 or 3, implying that respectively, $y\in P$ or $ord_y(q)=0$.
We conclude that $\nu_f$ satisfies condition \eqref{eq:3},
i.e. \textit{f} preserves an orbifold of type \textit{(iv)}.\\
Suppose now that $n_i=2$ for some $i \neq k, \frac k 2$,
then we have necessarily $\frac 1 2< \alpha_i \leq \frac 3 4$.
In fact, note that if $n_i=2$ then there is only another nonzero $n_j$ solving \eqref{eq:poles}, hence
$n_j \alpha_j=2(1-\alpha_i)$ and consequently, if $\alpha_i > \frac 3 4$,
we should have $n_j \alpha_j < \frac 1 2$ which is impossible.
Note that $\alpha_i\neq \frac{2}{3}$, otherwise we would have $n_i=3$. \\
We claim that there are no solutions of \eqref{eq:poles} with $n_i=2$ and
$\frac 1 2< \alpha_i < \frac 3 4$.\\
In fact, in this case \eqref{eq:poles} implies
that \textit{q} has exactly three poles, two of which having order $\alpha_i k$ and one having order
$\alpha_j k=2(1-\alpha_i)k$. However from \eqref{eq:bound}
we deduce that $deg_y(f)\leq3$ for any $y$ in the fiber of the poles of order $\alpha_i k$.
Clearly $deg_y(f)$ cannot be equal to 2 or 3, otherwise
we would have, respectively, $-ord_y(q) /k= 2 \alpha_i -1<\frac 1 2$ or
$-ord_y(q) /k= 3 \alpha_i -2<\frac 1 4$, which is impossible.
Therefore we obtain $deg_y(f)=1$, from which we deduce that the image of the ramification of \textit{f} consists of the other pole of \textit{q} (recall that \textit{f} maps any ramification point to a pole of \textit{q}).
The following simple computation shows that the image of
the ramification of a rational map of degree $d>1$ cannot
consist of one point. Let \textit{p} be this point and suppose that $f^{-1}(p)=e_1x_1+\dots+e_rx_r+y_1+\dots+y_s$, with $e_1+\dots+e_r+s=d$.
By the Riemann-Hurwitz formula, the ramification of a rational map of degree
\textit{d} has order $2d-2$, hence we obtain $2d-2=(e_1-1)+\dots+(e_r-1)=d-(r+s)$,
which is absurd since $d>1$.\\
Therefore we conclude that $n_i=2$ for some
$i \neq k, \frac k 2$ if and only if $\alpha_i=\frac 3 4 $
(note that it makes sense only if $k \equiv 0 \;(mod \,4)$). \\
In this case let $P=\{p_1,p_2\}$ be the set  poles of \textit{q}
of order $\frac 3 4 k$ and let \textit{p} be the other pole of
\textit{q}, necessarily of order $\frac k 2$. It is natural to define
\begin{equation}\label{244}
 \nu_f(x)=
\begin{cases}
  2, & \mbox{if } x=p; \\
  4, & \mbox{if } x \in P;\\
  1, & \mbox{otherwise}.
\end{cases}
\end{equation}
We already know that $\nu_f$ satisfies condition \eqref{eq:3} for $x=p$,
so we are left to show it holds also for $x\in P$.
In view of \eqref{eq:2}, given any $y \in f^{-1}(P)$,
the possible values for $deg_y(f)$ are 1, 2 or 4, meaning that
$\nu_f(y)$ is, respectively, equal to 4, 2, 1.
Thus we have $\nu_f(y)deg_y(f)=4$ in each case, so we conclude that \textit{f}
preserves an orbifold of type \textit{(v)}.\\
Finally we suppose that $n_i\leq 1,\;\forall i$.
It is clear that in this case we must have $\#\{i:n_i\neq 0\}=3$,
so we can rewrite equation \eqref{eq:poles}
in the form $\alpha + \beta + \gamma =2$,
with $\alpha,\beta,\gamma \in \Q$ satisfying $\frac 1 2 \leq \alpha<\beta<\gamma<1$.
We claim that this equation has only one solution,
which is $\alpha=\frac 1 2, \beta= \frac2 3 , \gamma= \frac5 6 $
(note that it makes sense only if $k \equiv0\;(mod\,6)$).
Suppose that $\alpha \neq \frac 1 2$ and observe that $\alpha<\frac 2 3$.
Thus, from \eqref{eq:bound},
we deduce that for any \textit{y} in the
fiber of the pole of order $\alpha k$ we have $deg_y(f)<3$.
Nevertheless, $deg_y(f)$ cannot be 2, since otherwise
we should have $-ord_y(q)/k<\frac 1 3$, which is impossible.
We conclude that the fiber of this pole must consist of exactly $d$
non-ramified different points, say $\{y_1,\dots,y_d\}$,
and for each of these points we should have $ord_{y_i}(q)=-\alpha k$,
but this leads to an absurd, since there is only one pole of such order.
We have reduced our equation to $\beta+\gamma=\frac 3 2 $,
with $\frac 1 2<\beta<\gamma<1$, but now the same argument
used for $\alpha$ shows that this is possible if and only if $\beta=\frac 2 3 $,
since otherwise the pole of order $\beta k$ would not be a branched point,
which leads us to an absurd. Calling $p_1,p_2,p_3$ the poles of \textit{q}
of order, respectively, $\alpha k,\beta k,\gamma k$, we define
\begin{equation}\label{236}
 \nu_f(x)=
\begin{cases}
  2, & \mbox{if } x=p_1; \\
  3, & \mbox{if } x=p_2;\\
  6, & \mbox{if } x=p_3;\\
  1, & \mbox{otherwise}.
\end{cases}
\end{equation}
We already know that $\nu_f$ satisfies condition \eqref{eq:3}
for $x=p_1,p_2$, so we need only to show that it holds also
for $p_3$. For any $y\in f^{-1}(p_3)$ the possible values for
$D=deg_y(f)$ are $D=1,2,3,6$ since from equation \eqref{eq:2}
we have $-ord_y(q)/k=1-\frac D 6$ and this can only be equal
to $0,\frac 1 2, \frac 2 3, \frac 5 6$.
It follows that $\nu_f(y)deg_y(f)=6$ in each case, so $\nu_f$ satisfies
condition \eqref{eq:3}, i.e.
\textit{f} preserves an orbifold of type \textit{(vi)}.
\end{proof}

\begin{remark}\label{remark1}
  We have shown that every solution of \eqref{eq:poles}
  with $\alpha_i \neq 1$ is such that
  $ \alpha_i=\left(1-\frac 1 n\right)$ for some
  $n=2,3,4 \mbox{ or } 6$, \textit{i.e.} $q$ may only have
  poles of order $k$ or $ \left(1-\frac 1 n\right)k$. \\
  In the latter case, observe that we can write equation \eqref{eq:poles} as
  \begin{equation}\label{eq:101}
    \sum\limits_{i \in I} \left(1- \frac{1}{e_i}\right)=2
  \end{equation}
  where $e_i$ are not necessarily distinct integers. \\
  Moreover observe that we have defined in each case
  $$
  \nu_f(x)=n \; \; \mbox{whenever  } ord_x(q)=-\left(1-\frac 1 n\right)k
  $$
  for any $n=1,2,3,4,6$, in such a way that
  $$
  \sum\limits_{x \in \PP} \left(1- \frac{1}{\nu_f(x)}\right)=2
  $$
  We conclude that the orbifolds \textit{(iii)-(vi)}
  of Lemma \ref{mainlemma} are also all the possible parabolic orbifolds on $\PP$. \\

\end{remark}

\section{Maps preserving a parabolic orbifold}
In this chapter we will discuss the main consequences
of Lemma \ref{mainlemma}. We have a rational map
$f:\PP \to \PP$ that has an invariant orbifold $\OO$,
i.e. there exists a map $\tilde{f}:\OO \to \OO$ such
that the following diagram commutes
\begin{equation}\label{diag1}
\begin{tikzcd}[column sep=2.5pc,row sep=2pc]
\OO \arrow{r}{\tilde{f}} \arrow{d}[swap]{p} & \OO \arrow{d}{p} \\
\PP \arrow{r}{f}   & \PP
\end{tikzcd}
\end{equation}
(here \textit{p} denotes the natural projection map),
with $\chi(\OO)=0$.\\
We will show that $\OO$ is the quotient of an elliptic curve
\textit{E} by the action of a (finite) group $G$ of automorphisms
of $E$ and that \textit{f} lifts to a morphism of $G$-torsors
$F: E \to E$, such that the following diagram commutes:
\begin{equation}\label{diag2}
\begin{tikzcd}[column sep=2.5pc,row sep=2pc]
{}      & E \arrow{r}{F} \arrow{d}{\pi} \arrow{ldd}[swap]{p\circ\pi}& E \arrow{d}[swap]{\pi} \arrow{rdd}{p\circ\pi}& \\
{}      & \OO \arrow{r}{\tilde{f}} \arrow{ld}{p}     & \OO  \arrow{rd}[swap]{p}   & \\
\PP \arrow{rrr}{f} &                      &           &    \PP
\end{tikzcd}
\end{equation}

\subsection{Construction of a torsor associated to a torsion line bundle}
We begin recalling some general facts concerning line bundles
over an orbifold $\OO$. We say that $\pi:L \to \OO$ is a torsion
line bundle of order \textit{n} if $L^{\otimes n}$ is trivial,
i.e. if there exists an isomorphism of line bundles on $\OO$
\begin{equation}\label{diag:trivial}
\begin{tikzcd}[column sep=2.5pc,row sep=2pc]
L^{\otimes n} \arrow[d,swap, "\pi"] \arrow[r, "\sim"] & \OO \times \C \arrow{ld}{p_1} \\
\OO &
\end{tikzcd}
\end{equation}
where $p_1$ denotes the projection on the first factor.\\
It is well known that a torsion line bundle over a compact manifold $X$
defines  a $\mu_n$-torsor over $X$, which is unique up to
an almost unique isomorphism, {\it i.e.} rather than
being unique, the isomorphisms
between any two such torsors form a principal homogeneous
space under $\mu_n$.
We are going to prove that this property still holds if $X$
is an orbifold whose underlying space is $\PP$.\\
Let us consider an orbifold $\OO$ modelled on $\PP$ whose
set of singular points is $\{x_1, \dots,x_r\} \subset \PP$,
each $x_i$ having finite weight $n_i$. This means that the
monodromy group of each $x_i$ is the group of $n_i$-th roots
of unity $\mu_{n_i}$ and arbitrarily small (non-space like)
neighborhood $U_i$ of $x_i$ in $\OO$ are described as follows:
we can choose a disk $\Delta_i$ centered at the origin such that
$U_i$ is the classifying champ, $[\Delta/\mu_{n_i}]$,
\cite[2.4.2]{3540657614}, for the action
$\Delta_i \leftleftarrows \Delta_i \times \mu_{n_i}$ of
$\mu_{n_i}$  by rotations on $\Delta_i$. \\
Let us recall that a line bundle \textit{L} on $\OO$ can be
described on  $U_i=\left[\Delta_i/\mu_{n_i}\right]\ni x_i$ as follows:
we have that $L\mathop{|}_{\Delta_i}$ is the trivial bundle, with action
determined by a representation $\rho_i: \mu_{n_i} \to \C^*$.
Note that $L\mathop{|}_{U_i}^{\otimes n_i}$ is the trivial bundle on
$U_i$ since such a representation has order dividing $n_i$.\\
If we set $Z=\coprod\limits_i p^{-1}(x_i)$ and $U=\OO \setminus Z$,
then \textit{L} is completely determined by the triple
\begin{equation}\label{triplet}
  \left(L\mathop{\mid}\nolimits_{Z}, L \mathop{\mid}\nolimits_{U}, \phi \right)
\end{equation}
where we have $L\mathop{\mid}_{Z}=\prod\limits_i (\rho_i: \mu_{n_i} \to  \C^*)$
and $\phi=\prod\limits_i \phi_i$, each
$\phi_i:L\mathop{\mid}_{U_i\setminus \{x_i\}} \to L\mathop{\mid}_{U} $ being
the gluing map of the bundle \textit{L}.\\
We have that \textit{L} is torsion of order $n$ if and only if
$L^{\otimes n}$ defines a line bundle on $\PP$ of degree 0,
since the following sequence
\begin{equation}\label{pic}
\begin{tikzcd}[column sep=2.5pc,row sep=2.5pc]
0 \arrow[r] & Pic(\PP) \arrow[r] & Pic(\OO) \arrow[r] & \prod\limits_i (\rho_i: \mu_{n_i} \to  \C^*)
\end{tikzcd}
\end{equation}
is exact.
Indeed if each representation $\rho_i$ is trivial, then
$L\mid_{U_i}$ is trivial so the maps $\phi$ are just
gluing with the trivial line bundle on the moduli of
$U_i$, {\it i.e.} the naive quotient of the $\mu_{n_i}$
action, identified with open subset of $\PP$, and so
the kernel is a line bundle on $\PP$.
We know, however, that $deg: Pic(\PP) \mathop{\to}\limits^{\sim} \Z$
is an isomorphism, so if $deg(L)=0$ and the order of
each representation $\rho_i$ divides $n$,
then from \eqref{triplet} and \eqref{pic} it follows
that \textit{L} is torsion of order dividing \textit{n}.

\begin{lemma}\label{torsor}
Let $\pi:L \to \OO$ be a $n$-torsion line bundle over an
orbifold $\OO$ whose underlying space is compact,
then associated to $L$ there is
a unique, up to isomorphism, $\mu_n$ torsor
$\mathcal{E} \subset L\xrightarrow{\pi} \OO$.
Better still the singular points of $\mathcal{E}$
lie over those of $\OO$, and if the underlying
space of $\OO$ is $\PP$ with $y$ a singular point
of $\mathcal{E}$ lying, in the above notation
\eqref{triplet},
over $x_i$ then the local monodromy of $\mathcal{E}$
at $y$ is the kernel of $\rho_i$.
\end{lemma}
\begin{proof}
From the exact sequence of sheaves on $\OO$
\begin{equation}\label{M101}
0\rightarrow \mu_n \rightarrow \mathbb{G}_m \xrightarrow{n}
\mathbb{G}_m
\rightarrow 0
\end{equation}
the long exact sequence in co-homology, since the
underlying space is compact, reads
\begin{equation}\label{M102}
0\rightarrow \mu_n \rightarrow \mathbb{C}^* \xrightarrow{n}
\mathbb{C}^*
\rightarrow H^1(\OO,\mu_n) \rightarrow Pic(\OO)
\xrightarrow{n} Pic(\OO)
\end{equation}
so, isomorphism classes of $\mu_n$ torsors are
exactly $n$-torsion line bundles.
In order to compute the singular points of $\mathcal{E}$,
we recall how to construct the torsor starting from the bundle.

In the first pace,  given a vector bundle $p:E \to \OO$ we can
form the tensor power of \textit{E}, $p':E^{\otimes n} \to \OO$
and we have a canonical map of bundles
\begin{equation}
\begin{tikzcd}[column sep=2.5pc,row sep=2pc]
E \arrow{r}{F_n} \arrow{d}[swap]{p} & E^{\otimes n} \arrow{ld}{p'} \\
\OO   &
\end{tikzcd}
\end{equation}
where the map $F_n$ sends an element $e \in E$ to its tensor
power $e^{\otimes n} \in E^{\otimes n}$.
Therefore if \textit{L} is a torsion line bundle of order
\textit{n} we have the following commutative diagram:
\begin{equation}\label{Michael1}
\begin{tikzcd}[column sep=2.5pc,row sep=2pc]
L \arrow{r}{F_n} \arrow{d}[swap]{\pi} & \OO \times \C \arrow{ld}{p_1} \\
\OO    &
\end{tikzcd}
\end{equation}
and for  any $\lambda \in \C^*$
our $\mu_n$ torsor is isomorphic to
$\mathcal{E}=F_n^{-1}(\OO \times \lambda)$.

To compute the monodromy at $y\mapsto x_i$
should $\OO$ have underlying space $\PP$, observe that
around $x_i$, \eqref{Michael1} corresponds to the map
of groupoids,
\begin{equation*}
\begin{tikzcd}[column sep=2.5pc,row sep=2pc]
(x,n(v)):=(x,v^n)  & \Delta \times \C & \arrow[l, shift left, "action"] \arrow[l, shift right, swap, "trivial"] \Delta \times \Gamma_p \times \C   \\
(x,v) \arrow[u, mapsto] & \Delta \times \C  \arrow{u}  & \arrow[l, shift left, "action"] \arrow[l, shift right, swap, "diagonal"]\Delta \times \Gamma_p \times \C \arrow{u}
\end{tikzcd}
\end{equation*}
in which the diagonal action is $(x,v)\to (x^{\gamma}, \rho_i(\gamma)v)$.

Hence we see that $F_n^{-1}(\Delta \times \lambda)$ is isomorphic to
the $\mu_n$ torsor $\Delta \times n^{-1}(\lambda)$ with the
stabilizer $\Gamma_i$ of $x_i$ acting diagonally, {\it i.e.}
$(x,l) \to (x^{\gamma}, \rho_i(\gamma)l)$, where $x \in \Delta$ and $l \in n^{-1}( \lambda)$.\\
We have that $\gamma \in Stab(x \times l) \Leftrightarrow x^{\gamma}=x \mbox{ and } \rho_p(\gamma)l=l$, which gives $x \in \Gamma_i$ and $\gamma \in \ker(\rho_p)$.
\end{proof}

\subsection{Holomorphic differentials on a parabolic orbifold}
In this section we show how the construction of Lemma \ref{torsor}
applies to the line bundle $\Omega_{\OO}$ of holomorphic differential
forms over a parabolic orbifold.\\
Around a non-space like point $x_i$, in the above notation \eqref{triplet},
we have that $\Omega_{\OO}\mathop{|}_{U_i}$ is the $\mu_{n_i}$-module
$\OO\mathop{|}_{\Delta_i}dz$, where $\OO\mathop{|}_{\Delta_i}$
denotes the sheaf of holomorphic functions on the disk, with action given by
\begin{equation}\label{diff}
 f(z)dz \to f(\theta z)\theta dz, \;\theta \in \mu_{n_i}
\end{equation}
Around a boundary point of $\OO$, \textit{i.e.} a singular point
with weight $\infty$, there is no orbifold structure, but still morally,
if not mathematically, the monodromy group of such a point is isomorphic to $\Z$. \\
We will denote by $(\OO, D)$ an orbifold $\OO$ with boundary $D$, \textit{i.e.}
$D \subset \OO$ is the set of singular points of weight $\infty$,
and by $\Omega_{\OO}(\log D)$ the sheaf of holomorphic differential
forms on $\OO$ with logarithmic poles on the boundary \textit{D}
(see \cite{3540051902}).
Moreover we will use the notation $f: (\OO_1,D_1) \to (\OO_2,D_2)$
for a map between orbifolds with boundary, meaning as usually that $f^{-1}D_2\subseteq D_1$.

\begin{lemma}\label{lemma:trivial}
Let $\OO=(\PP,\nu_f)$ (resp. $(\OO,D)=(\PP,\nu_f)$)
be a parabolic orbifold (resp. orbifold with boundary)
invariant for \textit{f} as in Lemma \ref{mainlemma},
and let \textit{q} the meromorphic section of
$\Omega_{\PP}^{\otimes k}$ such that $f^*q\,//\,q$. We have:
\begin{enumerate}
  \item $\Omega_{\OO}$ (resp. $\Omega_{\OO}(\log D)$)
  is a torsion line bundle of order $n:=lcm\{\nu_f(x): x \notin D\}$.
  \item If we denote by $p:\OO \to \PP$ the natural projection, then $\tilde{q}:=p^*q\in H^0(\Omega_{\OO}^{\otimes k})$
(resp. $H^0(\Omega_{\OO}(\log D)^{\otimes k})$
\end{enumerate}
\end{lemma}

\begin{proof}
1) From the definition of $n$ all the local representations \eqref{pic}
have order that divides $n$, so we need only to show that $deg(\Omega_{\OO})=0$ (resp. $deg(\Omega_{\OO}(\log D))=0$, but this follows from the fact that $deg(\Omega_{\OO})=-\chi(\OO)$ (resp. $deg(\Omega_{\OO}(\log D))=-\chi(\OO,D))$.\\
2) Let us consider a weighted point $x\in \OO$ of order \textit{n}.
The projection map $p:\OO \to \PP$ in the orbifold coordinate `\textit{s}'
around \textit{x} takes the form $s \to s^n$.\\
We have seen in \ref{remark1} that $p(x)$
is a pole of \textit{q} of order $\left(1-\frac{1}{n}\right)k$, so we have that
$$
\tilde{q}(z)= p^*\left(const \cdot z^{\frac{k}{n}}\left(\frac{dz}{z}\right)^k\right)=const\cdot ds^k
$$
Thus $\tilde{q}=p^*q$ is a holomorphic section of $\Omega_{\OO}^{\otimes k}$
(resp.  $\Omega_{\OO}(\log D)^{\otimes k}$)
satisfying $\tilde{f}^*(\tilde{q})\,//\,\tilde{q}$.
\end{proof}

\begin{lemma}\label{lemma3}
Let $\OO=(\PP,\nu_f)$ (resp. $(\OO,D)=(\PP,\nu_f)$) be an orbifold
(resp. orbifold with boundary) invariant for \textit{f} as in Lemma \ref{mainlemma}.
There exists an elliptic curve \textit{E}
(resp. the multiplicative group $\mathbb{G}_m$)
 which is a $\mu_n$-torsor over $\OO$
(resp. over $\OO\backslash D$)
 where $n$ is the order of torsion of $\Omega_{\OO}$ (resp. $\Omega_{\OO}(\log D))$.
 Moreover \textit{E} (resp.  the multiplicative group $\mathbb{G}_m$ viewed as
the manifold $(\PP, 0+\infty)$ with boundary)
is invariant for \textit{f}, i.e. there exists a morphism of torsors $F: E \to E$ (resp.  $F:(\PP, 0+\infty)\rightarrow (\PP, 0+\infty)$) such that the following diagram commutes:
\begin{equation}\label{diag}
\begin{tikzcd}[column sep=2.5pc,row sep=2pc]
(X, \partial) \arrow{r}{F} \arrow{d}[swap]{\pi} & (X, \partial) \arrow{d}{\pi} \\
(\OO, D) \arrow{r}{f}   & (\OO,D)
\end{tikzcd}
\end{equation}
where $(X,\partial)= (E, \emptyset)$ (resp. $(X,\partial)=(\PP, 0 + \infty)$).
\end{lemma}

\begin{proof}
We have seen in \ref{lemma:trivial} that if $n= lcm\{n_1,\dots,n_r\}$,
then $\Omega_{\OO}$ (resp. $\Omega_{\OO}(\log D)$) is a torsion line bundle
of order \textit{n} and we know from Lemma \ref{torsor} that it defines a unique $\mu_n$ torsor.
On the other hand, the representation defining $\Omega_{\OO}$ (resp. $\Omega_{\OO}(\log D)$) in a neighborhood of each $x_i$ is given by
the action of $\mu_{n_i}$ on differential forms \eqref{diff}
which is a faithful representation so by the
second part of \ref{torsor} it follows that $E:=\mathcal{E}$ is a Riemann surface (resp. Riemann surface with boundary) with Euler characteristic 0 i.e. \textit{E} is an elliptic curve if $\OO$ is compact, resp. $\mathbb{G}_m$
identified with $(\PP, 0+\infty)$ should the boundary be non-empty.\\

It remains to show that this implies the existence of a map of torsors $F: E \to E$, resp. with boundary, which makes the diagram \eqref{diag} commute. This follows from the fact that the isomorphism $f^*\Omega_{\OO} \mathop{\to}\limits^{\sim} \Omega_{\OO}$
(resp. $f^*\Omega_{\OO}(\log D) \mathop{\to}\limits^{\sim}
\Omega_{\OO}(\log D)$)
affords an isomorphism of $\mu_n$-torsors $f^*E \mathop{\to}\limits^{\sim} E$, resp. with boundary. \\

To see this, from \eqref{M101}-\eqref{M102}
we have a commutative diagram with exact rows:
\begin{center}
\begin{tikzcd}[column sep=2.5pc,row sep=2pc]
0\arrow{r} & H^1(\OO,\mu_n) \arrow{d}{f^*} \arrow{r} & H^1(\OO,\mathbb{G}_m) \arrow{d}{f^*} \arrow{r}{ n}& H^1(\OO,\mathbb{G}_m) \arrow{d}{f^*}  \\
0 \arrow{r} & H^1(\OO,\mu_n)  \arrow{r} & H^1(\OO,\mathbb{G}_m)  \arrow{r}{ n}& H^1(\OO,\mathbb{G}_m)
\end{tikzcd}
\end{center}
Thus just as $[\Omega_{\OO}]\, (\text{resp.}\,\, [\Omega_{\OO}(\log D)])\, \in H^1(\OO,\mathbb{G}_m)$ is fixed by $f^*$, so is $ [E]\in H^1(\OO,\mu_n)$, hence we have an isomorphism of $\mu_n$-torsors $E \mathop{\to}\limits^{\sim} f^*E$. Consequently
we obtain the following commutative diagram,
\begin{center}
\begin{tikzcd}[column sep=2.5pc,row sep=2pc]
E \arrow[r, "\sim"]  \arrow{rd} & f^*E \arrow{d} \arrow[r,"\tilde{f}"] & E \arrow[d] \\
{} & \OO \arrow[r,"f"] & \OO
\end{tikzcd}
\end{center}
and hence by composition we obtain a $\mu_n$-equivariant map $F: E \to E$.\\

\end{proof}
\begin{cor}\label{corollary}
Let  $f: \PP \to \PP$ be a rational map of degree $d>1$ which satisfies \textbf{Assumption} \ref{assum}.
Then \textit{f} preserves a parabolic orbifold (resp. a parabolic orbifold with boundary) which can be realized as a quotient of $\C$ (resp. of $\C \cup \{\infty\})$ by the action of a discrete subgroup of $Aut(\C)$.\\
The following table illustrates all the possibilities:\\

\begin{table}[htb]
\begin{tabular}{|c|c|}
  \hline
  Orbifold $\OO$ preserved by \textit{f}  & $\Gamma \subset Aut(\C)$ defining $\OO$ \\
  \hline
  $(\infty,\infty)$ & $\Z$, acting by translations \\
   \hline
  $(2,2,\infty)$& $<\Z, z \mapsto -z >$ \\
   \hline
  $(2,2,2,2)$ & $<\Z\oplus\Z\tau,z \mapsto -z>$; $\tau$ s.t. $\Im(\tau) >0$\\
   \hline
  $(3,3,3)$ & $<\Z[\zeta], z \mapsto \zeta z>$; $\zeta$ s.t. $\zeta^2 +\zeta +1=0$ \\
   \hline
  $(2,4,4)$ & $<\Z[i],z \mapsto i z>$, $i$ s.t. $i^2+1=0$ \\
   \hline
  $(2,3,6)$ & $<\Z[\zeta],z \mapsto -\zeta z>$, $\zeta$ as above  \\
  \hline
\end{tabular}
\end{table}
\end{cor}

\begin{proof}
We have seen that there exists a covering map $\pi:E \to \OO$ which can be viewed as the quotient map $E \to \left[E/\mu_n\right]$, for $n=2,3,4,6$, (resp. we have an isomorphism $\pi:\C^*\to \OO\setminus D$ in the case $(\infty,\infty)$, or a double cover $\pi:\C^* \to \OO\setminus D $ in the case $(2,2,\infty)$).
Consequently we have \begin{tikzcd}[cramped, sep=small]  \mu_{n}\arrow[r,hook] & Aut(E) \end{tikzcd}. Every such automorphism can be lifted to a linear map `$z \mapsto \alpha z$' on $\C$, the universal covering space of \textit{E}, which must satisfy $\alpha \Lambda = \Lambda$, where $\Lambda$ is the lattice defining \textit{E}. When $\alpha\in \mu_n$ a simple computation shows that $\Lambda = \Z[\mu_n]$ for $n=3,4,6$. In the case $n=2$ the condition above is empty, hence $\Lambda$ is generic.
\end{proof}

\begin{remark}\label{remark2}
The orbifolds listed in Corollary \ref{corollary} are the only one which can
be realized as quotients of an elliptic curve \textit{E} for the action of
a group of automorphisms of \textit{E}.
In fact it is well known that the group of automorphisms of an elliptic curve \textit{E}
is a finite cyclic group $G$ of order 2,4 or 6 (see \cite{0387962034}).\\
Consider the following exact sequence of algebraic groups:
\begin{center}
\begin{tikzcd}[column sep=2.5pc,row sep=2.5pc]
0 \arrow[r] & Aut^0(E) \arrow[r] & Aut(E) \arrow[r] & G \arrow[r] & 0
\end{tikzcd}
\end{center}
where $Aut^0(E)$ is the connected component of the identity in $Aut(E)$. We have that $E \mathop{\rightarrow}\limits^{\sim} Aut^0(E)$, where we identify $E$ with the subgroup of translations of $Aut(E)$.\\
Consider the action of $G$ on the elliptic curve, with quotient map $f: E \to E/G$.
It is a fact that the naive quotient $E/G$ is isomorphic to $\PP$, as the map \textit{f} is necessarily ramified and the Riemann-Hurwitz formula gives $\chi(E/G) >0$. \\
It follows that $\# Ram_f=2\# G$, so if the fiber of each $p\in E/G$ consists of $n_p$ distinct elements, each of order $e_p=\# Stab_G(p)$, we can write Riemann-Hurwitz as follows (note that we have $n_p e_p=\#G$),
 \begin{equation}\label{eq:quotient}
    \sum\limits_{p \in \PP}\left(1-\frac{1}{e_p}\right)=2
 \end{equation}
As discussed in \ref{remark1}, the only integer solutions of \eqref{eq:quotient} are given by $(iii)-(vi)$ of Lemma \ref{mainlemma}. Thus we conclude that $E/G$ can be endowed with the structure of a parabolic orbifold \textit{i.e.} an orbifold of type $(iii)-(vi)$.
\end{remark}
\subsection{Conclusions}
We conclude this chapter formulating at first a structural theorem for the maps $f$ with a parallel tensor $q$ and finally discussing the invariant nature of $q$.
\begin{theorem}
 Let  $f: \PP \to \PP$ be a rational map of degree $d>1$ which satisfies \textit{\textbf{Assumption}} \ref{assum}.\\
Then, up to conjugation with an element of $PGL_2(\C)$, $ f$ is induced by one of the automorphisms of $\C$ contained in \hyperref[table]{ Table \ref*{table} }, i.e. \textit{f} is obtained as quotient of such automorphisms under the action of a discrete group of automorphisms of the complex plane.\\
\begin{table}[htb]
\begin{tabular}{|c|c|}
  \hline
  Orbifold $\OO$ preserved by \textit{f}  & Automorphism of $\C$ inducing \textit{f} \\
  \hline
  $(\infty,\infty)$ & $z \mapsto nz$; $n\in\Z$ s.t. $|n|>1$ \\
   \hline
  $(2,2,\infty)$ & $z \mapsto nz+ \beta $; $n$ as above, $\beta=0,\frac{1}{2} $ \\
   \hline
   \multirow{2}{*}{$(2,2,2,2)$} & $z \mapsto \alpha z + \beta$; $\alpha$ an integer in an imaginary \\
  & quadratic field, $\beta \in E[2]$ \\
   \hline
  $(3,3,3)$ & $z \mapsto \alpha z + \beta$; $\alpha \in \Z[\zeta]$, $\beta=0, \frac{1}{3}(\zeta + 1), \frac 1 3 i \sqrt{3}$ \\
   \hline
  $(2,4,4)$ & $z \mapsto \alpha z + \beta$; $\alpha \in \Z[i]$, $\beta=0, \frac 1 2 (i+1)$ \\
   \hline
  $(2,3,6)$ & $z \mapsto \alpha z$; $\alpha \in \Z[\zeta]$  \\
  \hline
\end{tabular}
\caption{Dynamical systems admitting a parallel tensor}
\label{table}
\end{table}

\end{theorem}

\begin{proof}
This classification is originally due to the work of A. Douady and J. H. Hubbard \cite{MR1251582}. \\
From \ref{corollary} we know that \textit{f} preserves one of the orbifolds listed above, which is isomorphic to the quotient space $\left[E/\mu_n\right]$, for $n=2,3,4,6$. Moreover, as we have seen in \ref{lemma3}, \textit{f} lifts to a map $F:E \to E$ which commutes with the action of $\mu_n$. \\
It is well known, \cite{0387962034}, that $F$ is the composition of a translation with an endomorphism of $E$, so, as we have $End(\C/\Z[\mu_n]) \cong \Z[\mu_n]$, all endomorphisms are allowed since they commute with the action of $\mu_n$. However, the translation by any $Q \in E$ commutes with the action if and only if $Q$ is fixed by $\mu_n$. Consequently the only translations allowed are the solutions of
\begin{equation}\label{fixed}
  \theta z \equiv z \;(mod\,\Z[\theta])
\end{equation}
where $\theta$ is a primitive $n$-th root of unity. \hyperref[table]{ Table \ref*{table} } contains the solutions of \eqref{fixed} lying in the fundamental domain of the elliptic curve. For $n=2$ they consist of the subgroup $E[2]$ of 2-torsion, while there are only two non-zero fixed points for $n=3$, one for $n=4$ and there are no fixed points different from zero for $n=6$, since the order of the stabilizer at each point is given by its weight as a singular point in the orbifold structure.\\
Note that, modulo composition with a translation, \textit{f} lifts to an automorphism of the complex plane of the form $z\mapsto \alpha z$. Recall that when $\alpha \notin \Z$ we say that the elliptic curve $E$ has complex multiplication and if multiplication by a complex $\alpha$, (which must be an integer in some imaginary quadratic field, as $\alpha \Lambda \subset \Lambda$, see \cite{MR1251582}) is allowed in $E$, then the complex structure of $E$ is completely determined, i.e. we have $\Lambda \otimes_{\Z} \Q=\Q[\alpha]$.\\
Observe that we can compute explicitly the degree of \textit{f} in each case. In fact, from \eqref{diag} we deduce that $d=deg(f)=deg(F)$ and we know that, given $m \in \Z$, the ``multiplication by m'' $[m]:E \to E$, has degree $|m|$ if $E \cong \C^*$, while it has degree $m^2$ in the other cases, since it is equal to $\# \left(\Lambda[\frac 1 m]/\Lambda\right)$. If $E$ has complex multiplication $[\alpha]: z \mapsto \alpha z$, then as $deg([\alpha] \circ [\overline{\alpha}])=|\alpha|^4$, we obtain $deg([\alpha])= |\alpha|^2$ (see also \cite{MR1251582}).
\end{proof}

\begin{remarks}\
\begin{enumerate}[$\bullet$]
  \item Note that in each case we have written the maps $f:\PP \to \PP$ in the form $f=A\circ e$, where $A \in PGL_2(\C)$ corresponds to the translation by a fixed point on \textit{E} and \textit{e} is a rational map with the same degree of \textit{f}, corresponding to the endomorphim $z \mapsto \alpha z$ on $E$. Observe that the automorphism $A$ acts by permuting the poles of \textit{q}.\\
      A simple computation shows that $A$ has order 2 or 3. In fact if $E= \C/\Z[i]$ such a translation permutes the two fixed points of $z \to iz$, hence the corresponding map on $\PP$ has the form $z \mapsto 1/z$. If $E= \C/\Z[\zeta]$, the three fixed points are permuted cyclically by the translation, hence the corresponding map on $\PP$ has the form $z \mapsto \zeta z$.
  \item Cases (1) and (2) of Table \ref{table} are the most explicit:\\
In case (1) they are the maps $z\mapsto z^n$, since the exponential function is the universal covering map of $\C^*$.\\
In case (2) they are (up to sign) the Tchebycheff polynomials $P_n(z)$ defined by $P_n(cosz)=cos(nz)$, since the cosine function is the universal covering map of $\C^*$ which commutes with $z \mapsto -z$.
\item Note that in each case we have shown that (modulo multiplication by an element of $PGL_2(\C)$) the action of \textit{f} on $\PP$, which a priori is given by a semigroup, is globally equivalent to the action of some discrete group $G$ on $\C$, given by an extension of $\Z$ by $\Lambda$
    \begin{center}
\begin{tikzcd}[column sep=2.5pc,row sep=2.5pc]
0 \arrow[r] & \Lambda \arrow[r] & G \arrow[r] & \Z \arrow[r] & 0
\end{tikzcd}
\end{center}
\textit{i.e.} $G$ is the semidirect product $\Z \rtimes \Lambda$, with the obvious action.
\end{enumerate}
\end{remarks}

\noi We have seen in \ref{lemma:trivial} that $\tilde{q}=p^*q$ is a holomorphic section of $\Omega_{\OO}^{\otimes k}$, hence $\pi^*\tilde{q}$ is a constant multiple of $dz^{\otimes k}$. \\
It follows that the eigenvalue $\lambda$ such that $f^*q=\lambda q$ can be computed explicitly.
Referring to \hyperref[table]{ Table \ref*{table} }, in the first two cases we have simply $|\lambda |=deg(f)^k=d^k$
since the lifted map is multiplication by $m$ on $\C^*$, which has degree $|m|$, and clearly $[m]^*dz^{\otimes k}=m^kdz^{\otimes k}$. Finally, in the other cases, we have $d=|\alpha|^2$
from which we deduce $|\lambda|=d^{k/2}$.

\bibliography{bibbase}
\bibliographystyle{siam}

\end{document}